\documentclass[12pt,reqno]{amsart}

\usepackage{amsmath}
\usepackage{amsthm}
\usepackage{amssymb}
\usepackage{amsfonts}
\usepackage{latexsym}
\usepackage{hyperref}
\usepackage{cite}
\usepackage{xcolor}
\usepackage{thmtools}

\usepackage{todonotes}
\usepackage{comment}
\usepackage{multirow}
\usepackage{fullpage}
\usepackage{bbm}
\usepackage{algpseudocode}

\newcommand{\imod}[1]{\,(\textnormal{mod }#1)\,}
\newcommand{\lcm}[1]{\textnormal{lcm}(#1)}



\newtheorem{theorem}{Theorem}[section]

\newtheorem{lemma}[theorem]{Lemma}

\newtheorem{proposition}[theorem]{Proposition}

\theoremstyle{definition}

\newtheorem*{remark}{Remark}

\title{On the sum of a prime and a square-free number co-prime to any integer with at most two prime factors}
\author[E.~S.~Lee]{Ethan~Simpson~Lee}
\address{University of the West of England, School of Computing and Creative Technologies, Coldharbour Lane, Bristol, BS16 1QY} 
\email{ethan.lee@uwe.ac.uk}
\urladdr{\url{https://sites.google.com/view/ethansleemath/home}}

\author[R.~O'Clarey]{Rowan~O'Clarey}
\address{Universität Hamburg, Department of Mathematics, Bundesstrasse 55, 20146 Hamburg} 
\email{rowan.o.clarey@uni-hamburg.de}

\begin{document}

\begin{abstract}
Every natural number greater than $2$ can be written as the sum of a prime and a square-free number, and recent work has imposed additional divisibility conditions on the square-free number. We overcome limitations in these works to prove new results on square-free numbers co-prime to any integer with up to two prime factors, which make the expected asymptotic results explicit.
\end{abstract}

\maketitle

\section{Introduction} 

Let $\mathbb{S}$ denote the set of all positive square-free numbers and $\mathbb{S}_{k} = \{s\in\mathbb{S} : (s,k) = 1\}$, the set of square-free integers co-prime to $k$. 
Hathi and Johnston \cite[Cor.~4.1,~4.2]{HathiJohnston} have proved that for any fixed square-free integer $k > 1$, the following are true:
\begin{enumerate}
    \item If $k$ is odd, then every sufficiently large integer $n$ can be written as $n = p + \eta$ for some prime $p < n$ and $\eta \in \mathbb{S}_{k}$.
    \item If $k$ is even, then every sufficiently large even integer $n$ can be written as $n = p + \eta$ for some prime $p < n$ and $\eta \in \mathbb{S}_{k}$.
\end{enumerate}
When made explicit, results of this form may be viewed as weaker-but-provable analogues of Goldbach's conjecture. The more prime divisors $k$ has, the more challenging it is to prove explicit versions with a sharp range of $n$.

In this paper, we prove the following result, which makes these asymptotic results explicit when $k$ has at most two prime factors. The ranges of $n$ we compute are sharp.

\begin{theorem}\label{thm:main}
Fix any integer $k > 1$ with at most two prime factors.
\begin{enumerate}
    \item If $k$ is odd, then every integer $n \geq 36$ can be written as $n = p + \eta$ for some prime $p < n$ and $\eta \in \mathbb{S}_k$.
    \item If $k$ is even, then every even integer $n \geq 36$ can be written as $n = p + \eta$ for some prime $p < n$ and $\eta \in \mathbb{S}_k$.
\end{enumerate}
\end{theorem}

Our interest in this problem has been motivated by the advent of related explicit works in \cite{Dudek, FrancisLee, HathiJohnston}. In particular, \cite[Thm.~1]{Dudek} proves every integer $n > 2$ can be written as the sum of a prime $p < n$ and a square-free integer $\eta\in\mathbb{S}$, and the explicit results in \cite{FrancisLee, HathiJohnston} prove similar results restricting to $\eta\in\mathbb{S}_k$ for fixed choices of $k < 10^5$ with at most two prime factors. This limitation is inherited from the latest explicit bounds for the the error in the prime number theorem for arithmetic progressions (PNTAP), which only apply for $k \geq 10^5$ when $n$ is large. 
We overcome the $k < 10^5$ limitation in this paper, by incorporating new elementary techniques or the Brun--Titchmarsh theorem in those parts of the method where explicit bounds for the error term in the PNTAP are insufficient. 

The remainder of this paper is dedicated to proving \autoref{thm:main}.

\subsection*{Notation and structure}

Throughout this paper, $p$ always denotes a prime, $\mu$ is the M\"{o}bius function,
\begin{align*}
    T(n) &= \sum_{p\leq n} \mu^2(n-p) ,  
    &&T_q(n) = \sum_{\substack{p < n \\ (n-p,q) = 1}} \mu^2(n-p) , \\ 
    R(n) &= \sum_{p\leq n} \mu^2(n-p) \log{p} , \quad\text{and}
    && R_q(n) = \sum_{\substack{p < n \\ (n-p,q) = 1}} \mu^2(n-p)\log{p} .
\end{align*}
Explicit bounds for $T(n)$, $T_q(n)$, $R(n)$, and $R_q(n)$ are important ingredients in our proof of \autoref{thm:main}. 

In Section \ref{ssec:computations}, we verify \autoref{thm:main} for each $n\leq 4.81\cdot 10^9$ using computations. 
In Section \ref{sec:TR}, we introduce auxiliary bounds which underpin the analytic aspects of our proofs and apply these to establish the explicit bounds for $T(n)$, $T_q(n)$, $R(n)$, and $R_q(n)$ we require. The bounds we prove for $T(n)$, $T_q(n)$, $R(n)$, and $R_q(n)$ may also be of independent interest, as they are explicit bounds of the form $\gg n / \log{n}$ or $\gg n$. In Section \ref{sec:main_results}, we prove \autoref{thm:main} using these auxiliary results. 

\begin{comment}
\subsection*{Notation}

Throughout this paper, $p$ denotes a prime, $q_i$ denotes the $i$th prime, $\varphi$ denotes the Euler totient function, $\mu$ denotes the M\"{o}bius function, $\omega(k)$ is the number of distinct prime divisors $p\mid k$, 
\begin{equation*}
    T(n) = \sum_{p\leq n} \mu^2(n-p) , \quad\text{and}\quad
    R(n) = \sum_{p\leq n} \mu^2(n-p) \log{p} .
\end{equation*}
Here, $T(n)$ is the number of distinct representations of $n$ as the sum of a prime and a square-free number, $R(n)$ is a weighted variant of this counting function, and $N\#$ is a primorial. Our notation $N\#$ is non-standard, but is convenient for our purposes. In addition, 
\end{comment}

\subsection*{Acknowledgements} 

ESL thanks the Heilbronn Institute for Mathematical Research for their support. In addition, we thank Wilkie Hoare, Daniel Johnston, Chris Keyes, Pieter Moree, Timothy Trudgian, anonymous referees, and other colleagues for valuable feedback and insightful discussions.

\section{Computations}\label{ssec:computations}

Using computations, we prove Lemmas \ref{lem:comp_checks_mod_q}-\ref{lem:main2_p3} in this section. Each of these results are an important ingredient in our proof of \autoref{thm:main}. In fact, the computations in this section combine to verify \autoref{thm:main} for every $n \leq 4.81\cdot 10^9$. 

We outline the key ideas behind each algorithm in the proof of each lemma. The interested reader can read our Python code to establish each result at \href{https://github.com/EthanSLee/On-the-sum-of-a-prime-and-a-square-free-number-with-divisibility-conditions/tree/main}{\texttt{this link}}.

\begin{lemma}\label{lem:comp_checks_mod_q}
If $4\leq n \leq 4.81\cdot 10^9$ and $q > 10^5$ is prime, then $n$ may be written as the sum of {an odd} prime and a square-free number co-prime to $q$.
\end{lemma}

\begin{proof}
Any square-free integer $k \leq 10^5$ satisfies $(k,q) = 1$ for any prime $q > 10^5$. Therefore, we locate at least one representation $n = p + s$ for some {odd} prime $p$ and square-free $s \leq 10^5$ for each $n$ in the sub-intervals $[4,10^7)$ and $[\alpha\cdot 10^7, (\alpha+1)\cdot 10^7)$ for $\alpha\in [1,480]$ separately. {To optimise performance, our algorithm minimises the number of brute force computations by systematically accumulating a set of exceptional cases.}
\end{proof}

\begin{lemma}\label{lem:comp_checks_mod_q_smaller}
If $4 \leq n \leq 4.81\cdot 10^9$ and $2 < q < 10^5$ is prime, then $n$ may be written as the sum of {an odd} prime and a square-free number co-prime to $q$, unless $q = 3$, in which case $n = 11$ is an exception.
\end{lemma}

\begin{proof}
We verify this result for all $4 \leq n \leq 10^7$ by locating two primes $p_1 > p_2 > 2$ such that $n - p_i$ are square-free and $(n-p_1, n-p_2) \leq 2$. 
For all other $n$ under consideration, it has been verified in \cite{FrancisLee} that a prime $p$ exists such that $\mu^2(n-p) = 1$ and $(n-p,q) = 1$ for any prime $2 < q < 10^5$. In fact, for each $10^7 \leq n \leq 4.81\cdot 10^9$, the verification algorithm in \cite{FrancisLee} locates multiple primes $p \geq 9\,998\,279$ such that $n-p$ is square-free and the greatest common divisor of all such $n - p$ is $1$ or $2$, which suffices to verify the result. Since all primes $p \geq 9\,998\,279$ are odd, the result follows. \end{proof}

\begin{lemma}\label{lem:main2_p1}
Let $k$ be any odd square-free product with at most $N$ prime factors. 
\begin{enumerate}
    \item If $N \leq 2$, then every integer $36 \leq n \leq 10^5$ may be written as the sum of a prime and a square-free number co-prime to $k$. 
    \item If $N \leq 3$, then every integer $60 \leq n \leq 10^5$ may be written as the sum of a prime and a square-free number co-prime to $k$.
\end{enumerate}
\end{lemma}

\begin{proof}
Suppose $N\geq 2$ is an integer and let $\ell_1$, $\ell_2$, \dots, $\ell_{N+1}$ be a collection of pairwise co-prime square-frees, so that $(\ell_i,\ell_j) = 1$ for any $i\neq j$. This condition ensures that for any odd $k \geq 6$ with at most $N$ prime factors, $(k,\ell_i) = 1$ for at least one $i$. Therefore, to prove any even $n$ can be written as the sum of a prime and a square-free number co-prime to $k$ with at most $N$ prime factors, it suffices to locate a pairwise co-prime collection $\ell_1$, $\ell_2$, \dots, $\ell_{N+1}$ such that $n - \ell_i$ is prime for every $i\in\{1,2,\dots,N+1\}$. Using this observation, we verify the result with $N \in \{2, 3\}$ for every even integer $16 \leq n \leq 10^5$ by locating suitable $\ell_1$, $\ell_2$, $\ell_3$, and $\ell_4$ for each $n$. 

The preceding logic cannot cover odd $n$, since $n - p$ is even for any odd prime $p < n$. To adapt this logic to cover odd $n$ too, we need to locate even $\ell_1$, $\ell_2$, \dots, and $\ell_{N+1}$ such that if $p > 2$ is a prime dividing $\ell_i$, then $p\nmid \ell_j$ for any $i\neq j$. For any odd $k \geq 6$ with at most $N$ prime factors, $(k,\ell_i) = 1$ for at least one $i$. Therefore, to prove any odd $n$ can be written as the sum of a prime and a square-free number co-prime to $k$ with at most $N$ prime factors, it suffices to locate a suitable collection $\ell_1$, $\ell_2$, \dots, $\ell_{N+1}$ such that $n - \ell_i$ is prime for every $i\in\{1,2,\dots,N+1\}$. Using this observation, we verify the result with $N \in \{2, 3\}$ for every odd integer $n \leq 10^5$ by locating suitable $\ell_1$, $\ell_2$, $\ell_3$, and $\ell_4$ for each $n$.
\end{proof}

\begin{lemma}\label{lem:main2_p2}
If $k$ is any odd square-free product with at most three prime factors, then every even integer $10^5 < n \leq 4.81\cdot 10^9$ may be written as the sum of a prime and a square-free number co-prime to $k$.
\end{lemma}

\begin{proof}
If we can locate distinct primes $p_1$, $p_2$, $q_1$, $q_2$ such that $n = p_1 + p_2$ and $n = q_1 + q_2$, then at least one of $p_1$, $p_2$, $q_1$, or $q_2$ will be co-prime to any product of three or less primes. Therefore, there will be at least one representation of $n$ as the sum of a prime and a square-free number that is co-prime to any product of three or less primes. 
Using this observation, we verified the result for every even integer aside from $n = 740\,000\,138$ in the intervals $[10^5, 10^7]$ and $[\ell\cdot 10^7, (\ell + 1)\cdot 10^7]$ for every integer $1\leq \ell\leq 480$. Finally, the result is independently verified for $n = 740\,000\,138$ using the technique described in the proof of \autoref{lem:main2_p1}. That is, 
\begin{equation*}
    740\,000\,138 - 21, \quad
    740\,000\,138 - 235, \quad
    740\,000\,138 - 247, \quad\text{and}\quad
    740\,000\,138 - 391
\end{equation*}
are prime, and $21$, $235$, $247$, and $391$ are pairwise co-prime, so at least one of these representations will be admissible. 
\end{proof} 

\begin{lemma}\label{lem:main2_p3}
If $k$ is any odd square-free product with at most three prime factors, then every odd integer $10^5 < n \leq 4.81\cdot 10^9$ may be written as the sum of a prime and a square-free number co-prime to $k$. 
\end{lemma}

\begin{proof}
We verify the result for each odd $n$ in the sub-intervals $(10^5,10^6)$ and $I_{\alpha} = [\alpha\cdot 10^6, (\alpha+1)\cdot 10^6)$ separately, where $\alpha\in [1,4810]$ by detecting three distinct odd primes $p_1, p_2, p_3 < n$ such that each $n-p_i$ is square-free and $(n-p_i, n-p_j) = 2^{k}$ for some $k\geq 0$ and any $i\neq j$. Our algorithm minimises the number of brute force computations by systematically accumulating a set of exceptional cases. 
\end{proof}

\begin{remark}
The method we utilise in \autoref{lem:main2_p2} to efficiently handle even $n$ cannot be used to tackle odd $n$, as the sum of two odd primes is even. Therefore, we establish \autoref{lem:main2_p3} by extending the idea described in \autoref{lem:main2_p1}, which is more laborious and time inefficient. Some effort has been expended to reduce this run-time, which is feasible, but the motivated reader is invited to streamline our algorithm.
\end{remark}

\section{Bounds for the number of representations}\label{sec:TR}

To prove \autoref{thm:main} for every $n \geq 4.81\cdot 10^9$, we require explicit bounds for $T(n)$, $R(n)$, $T_q(n)$, and $R_q(n)$ such that $q\geq 3$ is prime. We prove the bounds we require (namely \autoref{thm:bounds_TR}, \autoref{cor:bounds_TR}, and \autoref{thm:bounds_TR_2}) in this section as follows.  
First, we import relevant auxiliary bounds in Section \ref{ssec:pnts}. In light of these results, we state and prove the bounds we require in Sections \ref{ssec:this_secA}, \ref{ssec:this_secB}, and \ref{ssec:this_secC}. 

\begin{remark}
It is likely that the intermediate computations to prove \autoref{thm:bounds_TR}, \autoref{cor:bounds_TR}, and \autoref{thm:bounds_TR_2} can be refined. However, the bounds we prove more than suffice for our purposes and any consequential refinements to the final results would be insignificant. 
\end{remark}

\begin{remark}
We write $C_{\text{Artin}} = 0.3739558136\ldots$ to denote Artin's constant, which appears often in what follows. This constant has been calculated to a high degree of accuracy by Wrench \cite{Wrench}. More recently, Moree \cite{Moree} developed an improved method for evaluating constants of the form  
\begin{equation*}
    \prod_{p > m} \left(1 - \frac{f(p)}{g(p)}\right) ,
\end{equation*}
where $f(t)$ and $g(t)$ are monic polynomials with integer coefficients such that $\deg{f} + 2 \leq \deg{g}$. In the special case $m=1$, $f(t)=1$, and $g(t) = t(t-1)$, this product coincides with $C_{\text{Artin}}$. Therefore, Moree's work facilitates a far more accurate computation of $C_{\text{Artin}}$. While the precision achieved in either paper is more than adequate for our purposes, we record this observation for completeness.
\end{remark}

\subsection{Prime number theorems}\label{ssec:pnts}

Suppose that $a$ and $q > 1$ are integers such that $(a,q) = 1$. We define 
\begin{equation*}
    \theta(x) = \sum_{p\leq x} \log{p}, \quad
    \pi(x,q,a) = \sum_{\substack{p\leq x\\p\equiv a\imod{q}}} 1 ,
    \quad\text{and}\quad
    \theta(x,q,a) = \sum_{\substack{p\leq x\\p\equiv a\imod{q}}} \log{p} .
\end{equation*} 
The Brun--Titchmarsh theorem (see \cite{MontgomeryVaughan}) tells us if $x > q$, then
\begin{equation}\label{eqn:BrunTitchmarsh}
    \pi(x,q,a) \leq \frac{2 x}{\varphi(q) \log(x/q)} .
\end{equation}
Further, the prime number theorem and relevant generalisations tell us that as $x\to\infty$, we have
\begin{equation*}
    \theta(x) \sim x
    \quad\text{and}\quad
    \theta(x,q,a) \sim \frac{x}{\varphi(q)} .
\end{equation*}
Explicit bounds for the error in these approximations are also required throughout this paper. To this end, we import the following results. 

\begin{theorem}[\text{Broadbent \textit{et al.} \cite{BroadbentEtAl}}]\label{thm:lethbridge_thm}
For $n > e^{20} \approx 3.59\cdot10^9$, we have
\begin{equation}\label{eqn:lethbridge_k}
      \lvert \theta(n) - n \rvert \leq 0. 375 \frac{n}{(\log{n})^3}.
 \end{equation}
\end{theorem}

\begin{theorem}[\text{Bennett \textit{et al.} \cite{BennettEtAl}}]\label{thm:BennettEtAl}
Let $q\geq 3$ be an integer and $a$ be an integer that is co-prime to $q$. There exist explicit constants $c_{\theta}(q)$ and $x_{\theta}(q)$ such that
\begin{equation*}
    \left| \theta(n,q,a) - \frac{n}{\varphi(q)} \right| < c_{\theta}(q)\frac{n}{\log{n}}
\end{equation*}
for all $x\geq x_{\theta}(q)$. Here, $c_\theta(q)\leq c_0(q)$ and $x_\theta(q)\leq x_0(q)$ with
\begin{align}
    c_0(q)=
    \begin{cases}
        \frac{1}{840} & \text{if } 3\leq q\leq 10^4 \\
        \frac{1}{160} & \text{if } q > 10^4
    \end{cases}\label{ctheta bounds}
\end{align}
and
\begin{align}
    x_0(q)=
    \begin{cases}
        8\cdot10^9 & \text{if } 3\leq q\leq 10^5 \\
        e^{0.03\sqrt{q}(\log{q})^3} & \text{if } q > 10^5
    \end{cases}\label{xtheta bounds}.
\end{align}
\end{theorem}

\subsection{Bounds for {$T(n)$} and {$R(n)$}}\label{ssec:this_secA}

First, we combine the bounds in \autoref{thm:lethbridge_thm} and \autoref{thm:BennettEtAl} to prove \autoref{thm:bounds_TR}, which provides the lower bounds for $T(n)$ and $R(n)$ we require. The proof of \autoref{thm:bounds_TR} is standard, but included to ensure this paper is self-contained. 

\begin{theorem}\label{thm:bounds_TR}
If $n \geq 4.81\cdot 10^9$, then
\begin{equation*}
    T(n) > \frac{0.32035 n}{\log{n}}
    \quad\text{and}\quad
    R(n) > 0.32035 n .
\end{equation*}
Moreover, if $n \geq e^{14\,476.991}$, then 
\begin{equation*}
    T(n) > \frac{0.37066 n}{\log{n}}
    %\begin{cases}
    %    \frac{0.74462 n}{\log{n}} &\text{if $2\mid n$} \\
    %    \frac{0.37066 n}{\log{n}} &\text{if $2\nmid n$}
    %\end{cases}
    \quad\text{and}\quad
    R(n) > 0.37066 n .
    %\begin{cases}
    %    0.74462 n &\text{if $2\mid n$} \\
    %    0.37066 n &\text{if $2\nmid n$}
    %\end{cases} .
\end{equation*}
\end{theorem}

\begin{proof}
We argue similar lines to \cite{Dudek, FrancisLee, HathiJohnston}, making minor modifications. 
To begin, note that if $c\geq 0$ is an integer and $0 < A < 1/2$, then
\begin{equation*}
    R(n) 
    = \Sigma_1 + \Sigma_2 + \Sigma_3 + \Sigma_4 ,
\end{equation*}
where 
\begin{align*}%\label{sigmas}
    \Sigma_1 &= \sum\limits_{\substack{a\leq c\\(a,n)=1}} \mu(a) \theta(n,a^{2},n),
    && \Sigma_2 = \sum\limits_{\substack{c < a\leq n^A\\(a,n)=1}} \mu(a)\theta(n,a^{2},n) ,\\
    \Sigma_3 &= \sum_{\substack{a\leq n^{A}\\(a,n) > 1}} \mu(a) \theta(n,a^{2},n) ,
    && \Sigma_4 = \sum\limits_{n^A < a\leq n^{1/2}} \mu(a)\theta(n,a^{2},n) .
\end{align*}
If $(a,n) > 1$, then $0\leq \theta(n,a^2,n) \leq \log{n}$, and hence
\begin{equation*}
    |\Sigma_3 | \leq n^{A}\log{n} .
\end{equation*}
Further, from $\theta(n,a^2,n) \leq \pi(n,a^2,n) \log{n} \leq \lceil n / a^2 \rceil \log{n}$, we have
\begin{align*}
    |\Sigma_4| 
    \leq \sum\limits_{n^A < a\leq n^{1/2}} \mu^2(a) \theta(n,a^{2},n)
    &< \left(n^{1-2A} - n^A + n^{1 - A}\right) \log{n} .
\end{align*}
Combing these observations, we have
\begin{equation}\label{eqn:lowerrr}
    | R(n) - \Sigma_1 - \Sigma_2 | \leq  \left(n^{1-2A} + n^{1 - A}\right) \log{n} .
\end{equation}

We choose $c = 316$, since $316^2 = 99\,856$ is the greatest square less than $10^5$, and the tables\footnote{The tables are available at \url{https://www.nt.math.ubc.ca/BeMaObRe/}.} in \cite{BennettEtAl} give explicit values for $c_\theta(q)$ and $x_\theta(q)$ in \autoref{thm:BennettEtAl} when $q\leq 10^5$. For $2\leq a\leq 316$, $x_\theta(a^2) \leq 4.81\cdot 10^9$, so it also follows from \autoref{thm:lethbridge_thm} and \autoref{thm:BennettEtAl} that if $n \geq x_0 = 4.81\cdot 10^{9}$, then we have
\begin{equation}\label{eqn:rf}
    \left| \theta(n,a^2,n) - \frac{n}{\varphi(a^2)} \right| \leq 
    \begin{cases}
        \frac{0.375 n}{(\log{n})^3} &\text{if $a = 1$,}\\
        \frac{c_\theta(a^2) n}{\log{n}} &\text{if $a > 1$.}
    \end{cases}
\end{equation}
Next, it follows from \eqref{eqn:rf} and $c = 316$ that if $n\geq x_0$, then
\begin{align*}
    \Bigg| \frac{\Sigma_1}{n} - \prod_{p\nmid n} \left( 1 - \frac{1}{p(p-1)}\right) + \sum\limits_{\substack{a > c \\ (a,n)=1}} \frac{\mu(a)}{\varphi(a^2)} \Bigg|
    %&\leq \frac{1}{n} \sum\limits_{\substack{a\leq c\\(a,n)=1}} \mu^2(a)E(n,a^{2},n) \\
    &\leq \frac{1}{\log{n}} \sum\limits_{2\leq a\leq c} c_\theta(a^2)\mu^2(a) + \frac{0.375}{(\log{n})^3} \\
    &< \frac{0.591}{\log{n}} + \frac{0.375}{(\log{n})^3} .
\end{align*}
It also follows from \eqref{eqn:BrunTitchmarsh} that if $a \leq n^A$, then
\begin{equation*}
    \Bigg| \theta(n,a^2,n) - \frac{n}{\varphi(a^2)} \Bigg| \leq \left(\frac{1+2A}{1-2A}\right) \frac{n}{\varphi(a^2)} ,
\end{equation*}
and hence
    \begin{align*}
        \Bigg|\frac{\Sigma_2}{n} - \sum\limits_{\substack{c < a\leq n^A\\(a,n)=1}} \frac{\mu(a)}{\varphi(a^{2})} \Bigg| 
        &\leq \frac{1+2A}{1-2A} \sum\limits_{\substack{c < a\leq n^A\\(a,n)=1}} \frac{\mu^2(a)}{\varphi(a^{2})} .
    \end{align*}
Combine these with \eqref{eqn:lowerrr} and the fact $(1+2A)/(1-2A) > 1$ on $A\in(0,1/2)$ to see
\begin{align*}
    \Bigg|\frac{R(n)}{n} - \prod_{p\nmid n} \left( 1 - \frac{1}{p(p-1)}\right)\Bigg| 
    &< \frac{1+2A}{1-2A} \sum\limits_{\substack{a>c\\(a,n)=1}}\frac{\mu^2(a)}{\varphi(a^{2})} + \frac{0.591}{\log{n}} + \frac{0.375}{(\log{n})^3} \\
    &\hspace{6cm} + \left(n^{-2A} + n^{- A}\right) \log{n} .
\end{align*}
Now, Ramar\'{e} has proved in \cite[Lem.~3.10]{Ramare} that if $Z>1$ is an integer, then
\begin{equation}\label{eqn:RamareIneq}
    \sum\limits_{a>Z}\frac{\mu^2(a)}{\varphi(a^{2})} \leq \frac{4}{Z} .
\end{equation}
Therefore, upon choosing $c = 316$ and $Z = 10^5$, we have
\begin{equation*}
    \sum\limits_{a>c}\frac{\mu^2(a)}{\varphi(a^{2})} 
    = \sum\limits_{a=c+1}^{Z}\frac{\mu^2(a)}{\varphi(a^{2})} + \sum\limits_{a > Z}\frac{\mu^2(a)}{\varphi(a^{2})} 
    < 0.00322 .
\end{equation*}
It follows that
\begin{align*}
    \Bigg|\frac{R(n)}{n} - \prod_{p\nmid n} \left( 1 - \frac{1}{p(p-1)}\right)\Bigg| 
    &< \left(\frac{1+2A}{1-2A}\right) 0.00322 \\
    &\qquad + \frac{0.591}{\log{n}} + \frac{0.375}{(\log{n})^3} 
    + \left(n^{-2A} + n^{- A}\right) \log{n} .
\end{align*}
This upper bound is minimised with $A = 0.35977$ at $n = 4.81\cdot 10^9$, so we assert $A = 0.35977$. Asserting this choice, if $n \geq 4.81\cdot 10^9$, then we have proved
\begin{equation}\label{eqn:Rn_approximation_optimised}
    \Bigg| \frac{R(n)}{n} - \prod_{p\nmid n} \left( 1 - \frac{1}{p(p-1)}\right) \Bigg| < 0.05361\dots . 
\end{equation}
Further, we have
\begin{align*}
    \prod_{p\nmid n} \left( 1 - \frac{1}{p(p-1)}\right)
    %&= \prod_{p} \left( 1 - \frac{1}{p(p-1)}\right) \prod_{p\mid n} \left( 1 + \frac{1}{p^2 - p - 1}\right) \\
    &= C_{\text{Artin}} \prod_{p\mid n} \left( 1 + \frac{1}{p^2 - p - 1}\right) .
\end{align*}
Combine \eqref{eqn:Rn_approximation_optimised} with this observation to see that for all $n \geq 4.81\cdot 10^9$, we have
\begin{equation*}
    \frac{R(n)}{n} 
    > C_{\text{Artin}} \prod_{p\mid n} \left( 1 + \frac{1}{p^2 - p - 1}\right) - 0.05362 
    > 0.32035 .
\end{equation*} 
With this, we have proved the result on $R(n)$ for $n \geq 4.81\cdot 10^9$. Repeat the preceding steps to also prove the result on $R(n)$ for $n \geq e^{14\,476.991}$, choosing $c = 316$, $A = 0.00162$, and $Z = 10^5$.

Finally, partial summation and $R(n) \geq \log{2}$, which is the sharpest bound we can use for all $n > 2$, implies that if $n > 2$, then
\begin{align}
    T(n) 
    = \frac{R(n)}{\log{n}} + \int_3^n \frac{R(t)}{t(\log{t})^2}\,dt 
    &\geq \frac{R(n)}{\log{n}} + \int_3^n \frac{\log{2}}{t(\log{t})^2}dt 
    > \frac{R(n)}{\log{n}} . \label{eqn:Tn_Rn_connection}
\end{align}
The final inequality holds and is not wasteful, because  
\begin{align*}
    0 < \int_3^n \frac{\log{2}}{t(\log{t})^2}\,dt 
    = \log{2}\bigg(\frac{1}{\log{3}}-\frac{1}{\log{n}}\bigg)
    < \frac{\log{2}}{\log{3}} 
    < 1 .
\end{align*}
Using \eqref{eqn:Tn_Rn_connection}, the lower bounds for $T(n)$ follows naturally from our lower bounds for $R(n)$.
\end{proof}

\subsection{Bounds for {$T_q(n)$} and {$R_q(n)$} when {$q>3$}}\label{ssec:this_secB}

Using \autoref{thm:bounds_TR}, we prove \autoref{cor:bounds_TR}, which gives the bounds we require for $T_q(n)$ and $R_q(n)$ when $q > 3$ is prime.

\begin{theorem}\label{cor:bounds_TR}
Let $q > 3$ be any fixed prime. If $n \geq 4.81\cdot 10^9$ and $(n,q) = 1$, then
\begin{equation*}
    T_q(n) > \frac{0.11978 n}{\log{n}}
    \quad\text{and}\quad
    R_q(n) > 0.11978 n .
\end{equation*}
\end{theorem}

\begin{proof}
To begin, note that
\begin{align*}
    R_q(n) 
    &= R(n) - \sum_{\substack{p<n \\ (n-p,q)=q}} \mu^2(n-p)\log{p} , 
\end{align*}
and hence 
\begin{align*}
    R_q(n) 
    &\geq R(n) - (\theta(n,q,n) - \theta(n,q^2,n)) 
    = R(n) - \sum_{j=1}^{q-1} \theta(n,q^2,n + jq^2) . 
\end{align*}
From this inequality, we prove the result for $q < 13$, $13 \leq q \leq 10^5$, and $q > 10^5$ independently.  

First, suppose $n \geq 4.81\cdot 10^9$ and $q \in \{5,7,11\}$. \autoref{thm:BennettEtAl} and \autoref{thm:bounds_TR} imply
\begin{align*}
    \frac{R_q(n)}{n} 
    &\geq 0.32 - \sum_{j=1}^{q-1} \frac{1}{\varphi(q^2)} - \frac{1}{840\log{n}} \sum_{j=1}^{q-1} 1 \\
    &= 0.32 - \frac{1}{q} - \frac{q-1}{840\log{n}}
    > 0.11978\dots .
\end{align*}
With this, the result for $R_q(n)$ is proved for every $q < 13$ under consideration. 
Next, suppose $n \geq 4.81\cdot 10^9$ and $13 \leq q \leq 10^5$. The Brun--Titchmarsh theorem \eqref{eqn:BrunTitchmarsh} and \autoref{thm:bounds_TR} imply
\begin{align*}
    \frac{R_q(n)}{n} 
    &\geq 0.32 - \sum_{j=1}^{q-1} \frac{2}{\varphi(q^2)} \left(1 - \frac{2\log{q}}{\log{n}}\right)^{-1}
    = 0.32 - \frac{2}{q} \left(1 - \frac{2\log{q}}{\log{n}}\right)^{-1} > 0.12017\dots . 
\end{align*}
With this, the result for $R_q(n)$ is proved for every $q \leq 10^5$. 
Finally, suppose $q > 10^5$ and note that $0.03 \sqrt{q} (\log{q})^3 \geq 14\,476.991$. If $n \geq e^{0.03 \sqrt{q} (\log{q})^3}$, then \autoref{thm:BennettEtAl} and \autoref{thm:bounds_TR} imply
    \begin{align*}
        \frac{R_q(n)}{n} 
        \geq 0.37 - \frac{\theta(n,q,n)}{n}
        &\geq 0.37 - \frac{1}{q-1} - \frac{1}{160\log{n}} \\
        &\geq 0.37 - \frac{1}{q-1} - \frac{1}{4.8 \sqrt{q} (\log{q})^3} 
        > 0.36998\dots .
    \end{align*}
Further, there are at most $\lfloor n/q \rfloor$ primes less than $n$ that are congruent to $n\imod{q}$. It follows from this and \autoref{thm:bounds_TR} that if $4.81\cdot 10^9 \leq n < e^{0.03 \sqrt{q} (\log{q})^3}$, then
    \begin{align*}
        \frac{R_q(n)}{n} 
        \geq 0.32 - \log{n} \sum_{j=1}^{q-1} \frac{1}{q^2} 
        &= 0.32 - \frac{\log{n}}{q}\left(1 - \frac{1}{q}\right) \\
        &> 0.32 - \frac{0.03(\log{q})^3}{\sqrt{q}}\left(1 - \frac{1}{q}\right)
        > 0.17523\dots. 
    \end{align*}
With this, the result for $R_q(n)$ is proved for every $q > 3$. Arguing similar lines to earlier, the bound for $T_q(n)$ for every $q > 3$ follows from the bounds for $R_q(n)$.
\end{proof}

\subsection{Bounds for {$T_3(n)$} and {$R_3(n)$}}\label{ssec:this_secC}

Finally, we use a refined argument to prove the following result, which gives the bounds we require for $T_q(n)$ and $R_q(n)$ when $q = 3$. %This result is another consequence of \autoref{thm:bounds_TR}. 

\begin{theorem}\label{thm:bounds_TR_2}
If $n \geq 4.81\cdot 10^9$ and $(n,3) = 1$, then
\begin{equation*}
    T_3(n) > \frac{0.09067 n}{\log{n}}
    \quad\text{and}\quad
    R_3(n) > 0.09067 n .
\end{equation*}
\end{theorem}

\begin{proof}
We update the arguments presented in \cite{FrancisLee, HathiJohnston} to handle $q=3$. To begin, let $0 < A < 1/2$ be a real parameter and note that 
\begin{align*}
    R_3(n) 
    = R(n) - \sum_{\substack{p<n \\ p\equiv n\imod{3}}} \mu^2(n-p)\log{p} 
    %&= R(n) - \sum_{\substack{p<n \\ p\equiv n\imod{3}}} \log{p} \sum_{a^2 \mid n - p} \mu(a) \\ 
    %&= R(n) - \sum_{a \leq \sqrt{n}} \mu(a) \theta(n,\lcm{3,a^2},n) \\
    &= S_1(n) + S_2(n) ,
\end{align*}
where
\begin{align*}
    S_1(n) &= \sum_{a \leq n^A} \mu(a) (\theta(n,a^2,n) - \theta(n,\lcm{3,a^2},n))
    \quad\text{and}\\
    S_2(n) &= \sum_{n^A < a \leq \sqrt{n}} \mu(a) (\theta(n,a^2,n) - \theta(n,\lcm{3,a^2},n)) .
\end{align*}
Note that if any square-free $a$ satisfies $3 \mid a$, then $\mu(a) (\theta(n,a^2,n) - \theta(n,\lcm{3,a^2},n)) = 0$, so
\begin{align*}
    S_1(n) 
    %&= \sum_{\substack{a \leq n^A \\ 3\nmid a}} \mu(a) (\theta(n,a^2,n) - \theta(n,3a^2,n))
    &= \sum_{\substack{a \leq n^A \\ 3\nmid a}} \mu(a) \sum_{j=1}^{2} \theta(n,3a^2,n + ja^2)
    \quad\text{and}\\
    S_2(n) &= \sum_{\substack{n^A < a \leq \sqrt{n} \\ 3\nmid a}} \mu(a) (\theta(n,a^2,n) - \theta(n,3a^2,n)) .
\end{align*}

First, we bound $S_1(n)$. To this end, recall the fact $\theta(n,a^2,n) \leq \log{n}$ when $(a,n) > 1$. It follows that
\begin{equation*}
    \Bigg| S_1(n) - \sum_{\substack{a \leq n^A \\ (a,3n) = 1}} \mu(a) \sum_{j=1}^{2} \theta(n,3a^2,n + ja^2) \Bigg| 
    \leq \sum_{\substack{a \leq n^A \\ (a,n) > 1}} \mu^2(a) \theta(n,a^2,n)
    \leq n^{A} \log{n}
\end{equation*}
Next, the tabulations in \cite{BennettEtAl, RamareRumely} and \autoref{thm:BennettEtAl} tell us that if $n\geq 4.81\cdot 10^9$ and $(b,n) = 1$, then
\begin{equation}\label{eqn:upper1}
    \Bigg| \theta(n,b,n) - \frac{n}{\varphi(b)} \Bigg| \leq 
    \begin{cases}
        2.072 \sqrt{n} &\text{if $3\leq b \leq 486$ and $4.81\cdot 10^9 \leq n \leq 8\cdot 10^{9}$,} \\
        \frac{n}{160\log{n}} &\text{if $486 < b \leq 10^5$ and $4.81\cdot 10^9 \leq n \leq 8\cdot 10^{9}$,} \\
        \frac{n}{160\log{n}} &\text{if $3 \leq b \leq 10^5$ and $n \geq 8\cdot 10^{9}$.}
    \end{cases}
\end{equation}
Therefore, if $n \geq 4.81\cdot 10^9$, then
\begin{align*}
    \Bigg| \sum_{\substack{3a^2 \leq 10^5 \\ (a,3n) = 1}} \mu(a) \sum_{j=1}^{2} \theta(n,3a^2,n + ja^2) - \frac{n}{2} \sum_{\substack{3a^2 \leq 10^5 \\ (a,3n) = 1}} \frac{\mu(a)}{\varphi(a^2)} \Bigg|
    &\leq \frac{n}{80\log{n}} \sum_{\substack{3a^2 \leq 10^5 \\ (a,3n) = 1}} \mu^2(a) \\
    &\leq \frac{n}{80\log{n}} \sum_{\substack{3a^2 \leq 10^5 \\ (a,3) = 1}} \mu^2(a)
    = \frac{84n}{80\log{n}} .
\end{align*}
Note the additional assumption $(n,3) = 1$ is required to ensure $(3a^2,n) = 1$ in this step. 
Further, \eqref{eqn:BrunTitchmarsh} implies that if $(a,3n) = 1$ and $a \leq n^A$, then
\begin{align*}
    \Bigg| \mu(a) \sum_{j=1}^{2} \theta(n,3a^2,n + ja^2) \Bigg|
    \leq \mu^2(a) \theta(n,a^2,n)
    &\leq \frac{2n\mu^2(a)}{\varphi(a^2)} \left(1 - \frac{2\log{a}}{\log{n}}\right)^{-1} \\
    &\leq \frac{2n\mu^2(a)}{(1 - 2A) \varphi(a^2)} .
\end{align*}
Combining these bounds with our earlier approximation for $S_1(n)$, we see that if $n \geq 4.81\cdot 10^9$, then
\begin{align*}
    \frac{1}{n} \Bigg| S_1(n) - \frac{1}{2} \sum_{\substack{3a^2 \leq 10^5 \\ (a,3n)=1}} \frac{\mu(a)}{\varphi(a^2)} \Bigg|
    &\leq 
    \frac{2}{1-2A} \sum_{\substack{\sqrt{10^5/3} < a \leq n^A \\ (a,3n) = 1}} \frac{\mu^2(a)}{\varphi(a^2)}
    + \frac{84}{80\log{n}} 
    + n^{A-1}\log{n} .
\end{align*}
We also have
\begin{align*}
    \Bigg| \sum_{\substack{3a^2 \leq 10^5 \\ (a,3n)=1}} \frac{\mu(a)}{\varphi(a^2)}
    - \prod_{\substack{p\nmid 3n}} \left(1 - \frac{1}{p(p-1)}\right) \Bigg|
    &\leq \sum_{\substack{3a^2 > 10^5 \\ (a,3n)=1}} \frac{\mu^2(a)}{\varphi(a^2)} .
\end{align*}
%and recall that \eqref{eqn:RamareIneq} implies
%\begin{equation*}
%    \sum_{\substack{a > n^A \\ (a,n)=1}} \frac{\mu^2(a)}{\varphi(a^2)} + \sum_{\substack{a \leq n^A \\ \lcm{3,a^2} > 10^5 \\ (a,n)=1}} \frac{\mu^2(a)}{\varphi(a^2)}
%    \leq \sum_{\substack{3a^2 > 10^5 \\ (a,n)=1}} \frac{\mu^2(a)}{\varphi(a^2)}
%    \leq \frac{4\sqrt{3}}{10^{5/2}} .
%\end{equation*}
Applying these observations in our earlier bound for $S_1(n)$, we see that if $n \geq 4.81\cdot 10^9$, then
\begin{align*}
    \frac{S_1(n)}{n} &\geq 
    \frac{1}{2} \prod_{\substack{p\nmid 3n}} \left(1 - \frac{1}{p(p-1)}\right) - \left(\frac{1}{2} + \frac{2}{1-2A}\right) \sum_{\substack{\sqrt{10^5/3} < a \leq n^A \\ (a,3n) = 1}} \frac{\mu^2(a)}{\varphi(a^2)} 
    - \frac{1}{2} \sum_{\substack{a > n^A \\ (a,3n)=1}} \frac{\mu^2(a)}{\varphi(a^2)} \\
    &\hspace{10cm} - \frac{84}{80\log{n}} - n^{A-1}\log{n} .
\end{align*}
Next, we bound $S_2(n)$ using standard arguments. If $n \geq 4.81\cdot 10^9$, then
\begin{align*}
    |S_2(n)| 
    \leq \log{n} \sum_{\substack{n^A < a \leq \sqrt{n}\\(a,3) = 1}} \mu^2(a) \left(\frac{n}{a^2} + 1\right) 
    &\leq n \log{n} \sum_{n^A < a \leq \sqrt{n}} \frac{\mu^2(a)}{a^2}
    + \log{n} \sum_{n^A < a \leq \sqrt{n}} \mu^2(a) \\
    &\leq n^{1-2A} \log{n}
    + \sqrt{n}\log{n} .
\end{align*}
Therefore, if $n \geq 4.81\cdot 10^9$, then \autoref{thm:bounds_TR} implies
\begin{align*}
    \frac{R_3(n)}{n} &\geq 
    \frac{1}{2} \prod_{\substack{p\nmid 3n}} \left(1 - \frac{1}{p(p-1)}\right) - \left(\frac{1}{2} + \frac{2}{1-2A}\right) \sum_{\substack{\sqrt{10^5/3} < a \leq n^A \\ (a,3n) = 1}} \frac{\mu^2(a)}{\varphi(a^2)} 
    - \frac{1}{2} \sum_{\substack{a > n^A \\ (a,3n)=1}} \frac{\mu^2(a)}{\varphi(a^2)} \\
    &\hspace{7.5cm} - \frac{84}{80\log{n}} - (n^{A-1} + n^{-2A} + n^{-\frac{1}{2}})\log{n} .
\end{align*}
Noting $2/(1-2A) > 1$ by definition and using \eqref{eqn:RamareIneq}, these bounds can be reshaped to the form
\begin{align*}
    \frac{R_3(n)}{n} 
    &\geq \frac{1}{2} \prod_{\substack{p\nmid 3n}} \left(1 - \frac{1}{p(p-1)}\right) - \left(\frac{1}{2} + \frac{2}{1-2A}\right) \sum_{\substack{3a^2 > 10^5 \\ (a,3) = 1}} \frac{\mu^2(a)}{\varphi(a^2)} \\
    &\hspace{8cm}- \frac{84}{80\log{n}} - (n^{A-1} + n^{-2A} + n^{-\frac{1}{2}}) \log{n} \\ 
    &\geq \frac{1}{2} \prod_{\substack{p\nmid 3n}} \left(1 - \frac{1}{p(p-1)}\right) - \left(\frac{1}{2} + \frac{2}{1-2A}\right) \frac{4\sqrt{3}}{10^{5/2}} 
    - \frac{84}{80\log{n}} - (n^{A-1} + n^{-2A} + n^{-\frac{1}{2}}) \log{n} .
\end{align*}
Since $(n,3) = 1$ by assumption, we also have
\begin{align*}
    \frac{1}{2} \prod_{\substack{p\nmid 3n}} \left(1 - \frac{1}{p(p-1)}\right)
    &= \frac{C_{\text{Artin}}}{2}  \prod_{\substack{p\mid 3n}} \left(1 + \frac{1}{p^2 - p - 1}\right) \\
    &= \frac{3 C_{\text{Artin}}}{5}  \prod_{\substack{p\mid n}} \left(1 + \frac{1}{p^2 - p - 1}\right)
    > \frac{3 C_{\text{Artin}}}{5} .
\end{align*}
Therefore, if $n \geq 4.81\cdot 10^9$, then
\begin{align*}
    \frac{R_3(n)}{n} 
    &\geq \frac{3 C_{\text{Artin}}}{5} - \left(\frac{1}{2} + \frac{2}{1-2A}\right) \frac{4\sqrt{3}}{10^{5/2}} 
    - \frac{84}{80\log{n}} - (n^{A-1} + n^{-2A} + n^{-\frac{1}{2}}) \log{n} .
\end{align*}
To maximise this lower bound, we assert $A = 0.18821$, from which it follows that if $n \geq 4.81\cdot 10^9$, then $R_3(n) \geq 0.09067 n$. With this, the advertised result for $R_3(n)$ is proved. The bound for $T_3(n)$ follows naturally from these bounds for $R_3(n)$ by arguing similar lines to earlier.
\end{proof}

\section{Proof of main result}\label{sec:main_results}

In this section, we prove \autoref{thm:main}. To begin, we prove the following lemmas (\autoref{lem:k_nmid_coprime} and \autoref{lem:k_nmid_not_coprime}), which are key ingredients in our proof. 

\begin{lemma}\label{lem:k_nmid_coprime}
Fix an odd square-free $k \geq 3$ that is prime or has two prime factors such that $k\geq 3\cdot 29$. If $n \geq 4.81\cdot 10^9$ such that $(n,k) = 1$, then there exists a prime {$2 < p < n$} such that $n - p \in \mathbb{S}_k$.
\end{lemma}

\begin{proof}
It suffices to prove the result for $k = q_1 q_2$, where $q_2 > q_1 \geq 3$ are primes such that $q_2 \geq 29$, as the result for every prime $k$ follows from \autoref{cor:bounds_TR} and \autoref{thm:bounds_TR_2}. 
For each $k = q_1 q_2$, it suffices to prove
\begin{equation}\label{eqn:sc}
    T_{q_1}(n) 
    > 1 + \sum_{\substack{p<n\\(n-p,q_2)=q_2}} 1
    \quad\text{or}\quad
    R_{q_1}(n) 
    > \log{2} + \sum_{\substack{p<n\\(n-p,q_2)=q_2}} \log{p} .
\end{equation}
We verify \eqref{eqn:sc} holds for all $n\geq 4.81\cdot 10^9$ and $q_2 \geq 29$, splitting our analysis into two cases. 

\medskip\noindent
\textsc{Case I:} To prove the result with $q_2 < 10^7$, observe that the Brun--Titchmarsh theorem \eqref{eqn:BrunTitchmarsh} and \autoref{cor:bounds_TR} or \autoref{thm:bounds_TR_2} certify that if $n\geq 4.81\cdot 10^9$ and $q_2 < 10^7$, then \eqref{eqn:sc} holds when
    \begin{equation*}
        \frac{0.09067 n}{\log{n}} > 1 + \frac{2n}{\varphi(q_2)\log{n}} \left(1 - \frac{\log{q_2}}{\log{n}}\right)^{-1} ,
    \end{equation*}
which is equivalent to the inequality
\begin{equation*}
    \left(0.09067 - \frac{\log{n}}{n}\right) (q_2-1) \left(1 - \frac{\log{q_2}}{\log{n}}\right) > 2 .
\end{equation*}
With this, the result is proved for every $n\geq 4.81\cdot 10^9$ when $29 \leq q_2 < 10^7$. 

\medskip\noindent
\textsc{Case II:} We prove the result with $q_2 \geq 10^7$ in two phases. First, recall there are at most $\lfloor n/q \rfloor$ primes $p < n$ satisfying $p\equiv n\imod{q}$. It follows from this observation and \autoref{cor:bounds_TR} or \autoref{thm:bounds_TR_2} that if $q_2 \geq 10^7$, then \eqref{eqn:sc} holds when
    \begin{equation*}
        \frac{0.09067 n}{\log{n}} > \frac{n}{q_2} + 1 ,
    \end{equation*}
which is equivalent to the inequality
\begin{equation*}
    \left(0.09067 - \frac{\log{n}}{n}\right) q_2 > \log{n} .
\end{equation*}
With this, the result is proved for every $4.81\cdot 10^9 \leq n \leq e^{0.09066 q_2}$ when $q_2 \geq 10^7$. 
Finally, note that $0.09066 q_2 \geq 0.03\sqrt{q_2}(\log{q_2})^3$ when $q_2 \geq 613\,445$. Therefore, if $n > e^{0.09066 q_2}$ and $q_2 \geq 10^7$, then \autoref{thm:BennettEtAl} and \autoref{cor:bounds_TR} or \autoref{thm:bounds_TR_2} imply \eqref{eqn:sc} holds when
\begin{equation*}
    0.09067 n > \log{2} + \frac{n}{q_2-1} + \frac{n}{160\log{n}},
\end{equation*}
which is equivalent to the inequality
\begin{equation*}
    \left(0.09067 - \frac{\log{2}}{n} - \frac{1}{160\log{n}}\right) (q_2 - 1) > 1 .
\end{equation*}
Since $n > e^{0.09066 q_2}$, it follows that \eqref{eqn:sc} holds when
\begin{equation*}
    \left(0.09067 - \frac{\log{2}}{e^{0.09066 q_2}} - \frac{1/160}{0.09066 q_2}\right) (q_2 - 1) > 1 ,
\end{equation*}
which holds for all $q_2 \geq 10^7$. With this, the result is proved for every $n\geq 4.81\cdot 10^9$ when $q_2 \geq 29$.
\end{proof}

\begin{lemma}\label{lem:k_nmid_not_coprime}
Fix an odd square-free $k \geq 3$ with at most two prime factors and let $n \geq 4.81\cdot 10^9$ such that $(n,k) > 1$. There exists a prime {$2 < p < n$} such that $n - p \in \mathbb{S}_k$.
\end{lemma}

\begin{proof}
If $k$ is prime or $k$ has two prime factors such that $(n,k)=k$, there exists a prime {$2 < p < n$} such that $n - p$ is square-free and $n - p \in \mathbb{S}_k$ when
\begin{equation}\label{eqn:SC1}
    R(n) > \log{2} + \sum_{q\mid k} \theta(n,q,n) ,
\end{equation}
wherein $\theta(n,q,n) \leq \log{n}$ for each prime $q\mid k$. In this case, it follows from this observation and \autoref{thm:bounds_TR} that \eqref{eqn:SC1} holds when 
\begin{align*}
    0.32 n > \log{2} + 2\log{n} ,
\end{align*}
which is clearly true for all $n$ under consideration. 
All that remains is to prove the result for $k = q_1 q_2$, where $q_i$ are distinct primes such that $(n,q_1) = 1$ and $(n,q_2) = q_2$. In this case, there exists a prime {$2 < p < n$} such that $n - p$ is square-free and $n - p \in \mathbb{S}_k$ when
\begin{equation}\label{eqn:SC1_alternative}
    R_{q_1}(n) > \log{2} + \theta(n,q_2,n) ,
\end{equation} 
wherein $\theta(n,q_2,n) \leq \log{n}$. This observation and \autoref{cor:bounds_TR} or \autoref{thm:bounds_TR_2} imply that \eqref{eqn:SC1_alternative} holds when 
\begin{align*}
    0.09067 n > \log{2} + \log{n} ,
\end{align*}
which is clearly true for all $n$ under consideration. With this, the result is proved.
\end{proof}

Next, we combine these lemmas with the computations from Section \ref{ssec:computations} to prove \autoref{thm:main1a}.

\begin{proposition}\label{thm:main1a}
Let $k > 1$ be odd and square-free.  
\begin{enumerate}
    \item If $k$ is prime, then every integer $n \geq 4$ can be written as $n = p + \eta$ for some odd prime $p < n$ and $\eta \in \mathbb{S}_k$.
    \item If $k$ has two prime factors such that $k\geq 3\cdot 29$, then every integer $n \geq 36$ can be written as $n = p + \eta$ for some odd prime $p < n$ and $\eta \in \mathbb{S}_k$.
\end{enumerate}
\end{proposition}

\begin{proof}
The computations from Section \ref{ssec:computations} verify the result for every $n \leq 4.81\cdot 10^9$. If $k\geq 3$ is prime or $k\geq 3\cdot 29$ has two prime factors and $n > 4.81\cdot 10^9$, then \autoref{lem:k_nmid_coprime} and \autoref{lem:k_nmid_not_coprime} prove the result. 
\end{proof}

Finally, we are well-positioned to complete the proof of \autoref{thm:main}. 

\begin{proof}[Proof of \autoref{thm:main}]
\autoref{thm:main1a} proves the result for every prime $k > 1$, 
the explicit results in \cite{HathiJohnston} verify the result for every $k < 3\cdot 29$ with two prime factors, and 
\autoref{thm:main1a} proves the result for every odd $k \geq 3\cdot 29$ with two prime factors. 
Therefore, it suffices to prove the result for \textit{even} square-free integers $k \geq 3\cdot 29$ with two prime factors. 

To this end, we write $k = 2q$ for some prime $q \geq 3$. It follows from \autoref{thm:main1a} that for every $n \geq 36$ there exists a prime $2 < p < n$ such that $n - p \in \mathbb{S}_q$. If $n \geq 36$ is even, then we also have $(n-p,k) = 1$, because $n - p$ is odd. Therefore, for every even $n \geq 36$, there exists a prime $2 < p < n$ such that $n - p \in \mathbb{S}_k$. With this, the result is proved.
\end{proof}

\bibliographystyle{amsplain}
\bibliography{refs}
\end{document}